\documentclass[a4paper,12pt]{amsart}
 
\usepackage{amssymb, amsmath, amsthm, amsfonts}
\usepackage{extsizes}
\usepackage{graphicx,color}    
 \usepackage{changes,soul}    

\usepackage{url} 

\newtheorem{theorem}{Theorem}[section]
\newtheorem{corollary}[theorem]{Corollary}
\newtheorem{proposition}[theorem]{Proposition}
\newtheorem{lemma}[theorem]{Lemma}

\theoremstyle{definition}
\newtheorem{definition}[theorem]{Definition}
\newtheorem{example}[theorem]{Example}

\DeclareMathOperator{\sech}{sech}

\def\r{\mathbb R}

\def\d{\mathsf d}
\def\d{\mathbb D}
\begin{document}

\title{Generalization of the catenary in the Dual plane}
 
\author{Muhittin Evren Aydın}
\address{Department of Mathematics, Faculty of Science, Firat University, Elazig,  23200 Turkey}
\email{meaydin@firat.edu.tr}
\author{Rafael L\'opez}
\address{Departamento de Geometr\'{\i}a y Topolog\'{\i}a Universidad de Granada 18071 Granada, Spain}
\email{rcamino@ugr.es}
\begin{abstract}
In this paper, we study a dual analogue of the classical catenary within the class of admissible curves in the dual plane $\d^2$. We introduce $\alpha$-catenaries in $\d^2$ as stationary points of a potential energy functional, where $\alpha\in \r$ is a real parameter. We derive the corresponding Euler-Lagrange equations and obtain explicit equations of these curves for specific values of $\alpha$. Furthermore, we establish a geometric characterization of $\alpha$-catenaries in terms of their curvature and unit normal vector field.
\end{abstract}

\maketitle
\section{Introduction and Formulation of The Problem}

The problem of describing the shape of a hanging inextensible chain suspended under gravity and fixed at its endpoints is one of the classical problems in mathematical physics. Its solution, known as the catenary, was obtained toward the end of 17th century by the independent contributions of Hooke, Leibniz, Huygens, Bernoulli, and others.

From a variational viewpoint, the standard expression of the catenary is as follows. Let $(x,y)$ be the canonical coordinates in the Euclidean plane $\r^2$ and consider a curve $y(x)$, $x\in[a,b]$, in $\r^2$. In the upper half-plane $y>0$, the shape of the chain is the curve which minimizes the gravitational potential energy among all curves connecting the given endpoints. This curve is the catenary, given by
$$
y(x)=\frac1c\cosh(cx+m), \quad c,m\in \r, c>0.
$$
Extensions of catenaries in various other ambient spaces can be found in \cite{dl1,dl2,dl3,lo1,lo2,lo3}.

In this paper, we generalize the notion of the catenary to the dual plane $\d^2$. Recall that the study of dual numbers was initiated by William Kingdon Clifford (1845--1879) \cite{clif}. A dual number is an expression of the form $a+\varepsilon b$, where $a,b\in \r$ are  called its real and dual parts, respectively, and $\varepsilon$ is a nilpotent element satisfying $\varepsilon^2=0$ with $\varepsilon\neq0$. An $n$-dimensional dual space $\d^n=\r^n \times \varepsilon \r^n$ is a module over the commutative ring of dual numbers $\d$ with unity \cite{vel}.

Let $\r^3$ be the Euclidean space. Eduard Study (1862-1930) \cite{stud} established a correspondence between lines in $\r^3$ and dual vectors in $\d^3$,  which gave rise to the systematic investigation of the kinematics of curves. This correspondence provided the algebraic basis for describing rigid-body motions and screw displacements via dual numbers, thereby initiating the systematic development of spatial kinematics.

Denote by $\langle,\rangle$ the Euclidean metric of the space $\r^n$ and by $|\cdot|$ the associated norm. We use the same notation $\langle,\rangle$ for the dual metric on $\d^n$ , obtained by extending the Euclidean metric under the condition $\varepsilon^2=0$. Let $\mathbf{u}\in \d^n$ be a dual vector of the form $\vec{u}+\varepsilon \vec{v}$, where $\vec{u},\vec{v}\in \r^n$ with $\vec{u}\neq 0$. The dual norm of $\mathbf{u}$ is defined by $|\mathbf{u}|=|\vec{u}|+\varepsilon \frac{\langle \vec{u},\vec{v}\rangle }{|\vec{u}|}$ (see \cite{vel}).

Consider a real-parametrized curve $\gamma:I\subset \r \to\d^n$ given by $\gamma(t)=\psi(t)+\varepsilon \phi (t)$, where $\psi,\phi:I\to \r^n$ are smooth curves. The derivative of $\gamma$ with respect to the real parameter $t$ is simply $\gamma'(t)=\psi'(t)+\varepsilon \phi' (t)$ (see \cite{adg,dim}). 

We call $\gamma=\gamma(t)$ {\it admissible} if its real part $\mbox{Re}(\gamma)=\psi$ is a regular curve in $\r^n$ and the velocities of the real and dual parts are orthogonal, i.e. $\langle \psi',\phi'\rangle=0$ on $I\subset \r$. For every admissible curve, it follows that $|\gamma'|=|\psi'|$ and therefore the norm $|\gamma'|$ is always real-valued.

Recently, the second author \cite{lo} has proved that a curve $\gamma(t)$ is admissible if and only if it can be reparametrized by arc-length, that is, by a real variable $s$ satisfying $|\gamma'(s)|=1$ for all $s$. Notice that in the Euclidean setting, the regularity of a curve is both necessary and sufficient condition for the existence of an arc-length reparametrization. In contrast, in the dual setting, regularity remains necessary but is no longer sufficient: a dual curve must also satisfy the orthogonality condition between the velocities of its real and dual parts.

Existence theorems for admissible curves in the dual setting were established in \cite{lo}, where a complete classification of curves with constant curvature and torsion was also obtained. See also \cite{ace}.

The aim of this paper is to extend the notion of the catenary to the class of admissible curves in
$\d^2$. For this, let $\vec{u}\in \r^2$ be a fixed unit vector and consider the half-plane 
$$\r^2_+(\vec{u})=\{p\in \r^2:\langle p,\vec{u}\rangle >0 \}.$$
Denote by $\mathbf{u}\in \d^2$ a fixed unit dual vector of the form $\mathbf{u}=\vec{u}+\varepsilon\vec{v},$  where $\vec{v}\in \r^2$. Consider a smooth admissible curve 
 $$\gamma:[a,b]\subset\r\to \d^2,\quad  \gamma(t)=\psi(t)+\varepsilon \phi (t),$$
where  $\psi:[a,b]\to \r^2_+(\vec{u})$ and $\phi:[a,b]\to \r^2$. We define the potential $\alpha$-energy of $\gamma$ in the direction of $\mathbf{u}$ as the first-order functional given by
\begin{equation}\label{endu}
\mathcal{E}[\gamma]=\int_a^b \langle \mathbf{u},\gamma(t)\rangle^\alpha |\gamma'(t)|\, dt,
\end{equation}
where $\alpha$ is a real constant. The integrand expands as
\begin{equation}\label{endu2}
\langle \mathbf{u},\gamma(t)\rangle^\alpha=\langle\vec{u},\psi(t)\rangle^\alpha +\varepsilon\alpha(\langle \vec{u},\phi(t)\rangle+\langle \vec{v},\psi(t)\rangle)\langle\vec{u},\psi(t)\rangle^{\alpha-1} .
\end{equation}
Since $\langle\vec{u},\psi(t)\rangle>0$ for all $t\in [a,b]$, the expression $\langle \mathbf{u},\gamma(t)\rangle^\alpha$ is well defined and so is the functional $\mathcal{E}[\gamma]$.

Therefore, we introduce the following definition:

\begin{definition}
Let $\alpha\in \r$. An admissible curve $\gamma:[a,b]\to\d^2$ is called an $\alpha$-catenary in $\d^2$ if it is a stationary point of the energy $\mathcal{E}[\gamma]$ given in \eqref{endu}. In particular, when $\alpha=1$, the curve $\gamma$ is called a catenary in $\d^2$.
\end{definition}

The organization of the paper is as follows. In Section \ref{sec2}, we characterize $\alpha$-catenaries in $\d^2$ by splitting the first variation of the energy into its real and dual parts (Proposition \ref{p21}). We then derive the corresponding Euler-Lagrange equations (Proposition \ref{p22}) and obtain explicit classifications for the specific cases $\alpha \in \{-1, 0, 1\}$ (Propositions \ref{p23}-\ref{p25}). We also introduce the notion of the reversed catenary and prove that it provides a natural construction of $\alpha$-catenaries in $\d^2$ directly from their Euclidean counterparts (Theorem \ref{t27}). Finally, in Section \ref{sec3}, we establish a geometric characterization of $\alpha$-catenaries in $\d^2$, relating their dual curvature to their dual normal vector (Theorem \ref{t32}), which beautifully mirrors the classical Euclidean case. 

\section{The Euler-Lagrange equation and examples}\label{sec2}

In this section, we compute the first variation of the energy defined in \eqref{endu}. For this, let $\gamma:[a,b]\to\d^2$ be an admissible curve whose real and dual parts are respectively given by
$$
\psi:[a,b]\to\r^2_+(\vec{u}), \quad \phi:[a,b]\to\r^2. 
$$ 
Denote by $\mathbf{u}=\vec{u}+\varepsilon\vec{v}$ a unit dual vector, where $|\vec{u}|=1$ and $\langle \vec{u},\vec{v}\rangle=0$. Since $\gamma$ is admissible, it follows that $|\gamma'(t)|=|\psi'(t)|$ for all $t$. 

We first show how to obtain \eqref{endu2}. On the left side of \eqref{endu2}, the term $ \langle \mathbf{u},\gamma(t)\rangle
$ expands as
$$
\langle \mathbf{u},\gamma(t)\rangle=\langle\vec{u},\psi(t)\rangle +\varepsilon(\langle \vec{u},\phi(t)\rangle+\langle \vec{v},\psi(t)\rangle).
$$  
Let $\mathbf{x}=x+\varepsilon y \in U \subset \d$ be a dual variable and $f:U\to \d$ a smooth function, $f=f(\mathbf{x})$. Then its first-order Taylor expension in the dual sense gives $f(\mathbf{x})=f(x)+\varepsilon y f'(x)$. Now, viewing $f(\mathbf{x})=\mathbf{x}^\alpha$ as a power function and setting
$$
\mathbf{x}=\langle \mathbf{u},\gamma(t)\rangle, \quad x=\langle\vec{u},\psi(t)\rangle, \quad y=\langle \vec{u},\phi(t)\rangle+\langle \vec{v},\psi(t)\rangle,
$$
we immediately obtain \eqref{endu2}.

Next, the energy \eqref{endu} can be written as
$$
\mathcal{E}[\gamma]=\mathcal{E}_0[\psi]+\varepsilon\mathcal{E}_1[\psi,\phi],
$$
where
\begin{equation} \label{endu1}
\begin{array}{l}\vspace{.5cm}
\mathcal{E}_0[\psi]=\int_a^b \langle \vec{u},\psi(t)\rangle^\alpha |\psi'(t)|dt,\\
\mathcal{E}_1[\psi,\phi]=\alpha \int_a^b (\langle \vec{u},\phi(t)\rangle+\langle \vec{v},\psi(t)\rangle)\langle\vec{u},\psi(t)\rangle^{\alpha-1} |\psi'(t)|dt.
\end{array}
\end{equation}

In the following proposition, we provide a characterization of $\alpha$-catenaries in $\d^2$.

\begin{proposition}\label{p21}
Let $\alpha\in \r$, $\alpha\neq0$. An admissible curve $\gamma:[a,b]\to\d^2$, $\gamma=\psi+\varepsilon \phi$, is an $\alpha$-catenary in $\d^2$ if and only if $ \psi$ is an $\alpha$-catenary in $\r^2$ and $\mathcal{E}_1'[\psi,\phi]=0$.
\end{proposition}
\begin{proof}
Let $\rho \in (-\delta,\delta)\subset \r$. Consider the smooth variation $\gamma_\rho:[a,b]\to \d^2$ given by
$$
\gamma_\rho(t)=\psi_\rho(t)+\varepsilon \phi_\rho(t), \qquad \gamma_0=\gamma.
$$
Then, we have
$$
\mathcal{E}[\gamma_\rho]
=\mathcal{E}_0[\psi_\rho]+\varepsilon \mathcal{E}_1[\psi_\rho,\phi_\rho].
$$
The first variation of $\mathcal{E}$ at $\gamma$ is defined by
$$
\mathcal{E}'[\gamma]=\left.\frac{d}{d\rho}\right|_{\rho=0}\mathcal{E}[\gamma_\rho].
$$
Since the differentiation of dual-valued functionals with respect to a real parameter is linear, we obtain
\begin{equation}\label{eq:firstvar-split}
\mathcal{E}'[\gamma]
=\mathcal{E}_0'[\psi]+\varepsilon\,\mathcal{E}_1'[\psi,\phi].
\end{equation}
Hence, if $\gamma$ is a stationary point of $\mathcal{E}[\gamma]$, then $\mathcal{E}'[\gamma]=0$, which means that $\mathcal{E}_0'[\psi]=0$ and $\mathcal{E}_1'[\psi,\phi]=0$. In other words, $\gamma(t)$ is an $\alpha$-catenary in $\d^2$ if and only if
\begin{equation}\label{eqsplit}
\mathcal{E}_0'[\psi]=0, \quad
\mathcal{E}_1'[\psi,\phi]=0.
\end{equation}
It is known that the stationary points of the energy $\mathcal{E}_0[\psi]$ correspond to the $\alpha$-catenaries in $\r^2$ (see \cite{die}). 
 
\end{proof}

Next, we derive the Euler-Lagrange equation corresponding to the energy $\mathcal{E}_1$ given in \eqref{endu1}. After a suitable change of coordinates, we may assume that a unit dual vector $\mathbf{u}$ is given by
$$\mathbf{u}=\vec{u}+\varepsilon \vec{v},\quad \vec{u}=(0,1), \vec{v}=(v,0).$$

Let $\gamma(x)\subset \d^2$, $x\in [a,b]\subset\r$, be an admissible curve written in the form 
\begin{equation}\label{g}
\begin{split}
\gamma(x)&=\psi(x)+\varepsilon \phi(x),\\
\psi(x)&=(x,y(x)),\\
\phi(x)&=(w(x),z(x)),
\end{split}
\end{equation}
where $w,y,z$ are smooth functions on the interval $[a,b]$. 

We highlight two observations. First, the function $w$ cannot be constant. Indeed, if $\phi(x)=(w_0,z(x))$, $w_0\in\r$, then the orthogonality condition between $\psi$ and $\phi$ gives $y'(x)z'(x)=0$ on $[a,b]$. If $z(x)$ were constant, then $\phi$ would reduce to a single point, which is not our case. Hence $z'(x)\neq 0$ and therefore $y'(x)=0$, which is a contradiction. 

Second, even if $w(x)$ is not constant, we cannot reparametrize $\phi$ as $\phi(x)=(x,z(x))$ because this would imply a reparametrization of the curve $\psi$, which is already fixed in \eqref{g}.
 
In terms of the coordinate functions, the orthogonality condition between the curves $\psi$ and $\phi$ is equivalent to  
\begin{equation}\label{orto}
w'+y'z'=0.
\end{equation}

If $\psi(x)$ is an $\alpha$-catenary in $\r^2$, then the function $y(x)$ satisfies
\begin{equation}\label{y11}
\frac{y''}{1+y'^2}=\frac{\alpha}{y}.
\end{equation}
Solving this differential equation yields
\begin{equation}
1+y'^2=c^2y^{2\alpha}, \label{y12}
\end{equation}
where $c>0$ is a real constant. By differentiating, one obtains 
\begin{equation}\label{y13}
y''=\alpha c^2 y^{2\alpha-1}.
\end{equation}

In the next result, we characterize  all $\alpha$-catenaries in $\d^2$.
\begin{proposition}\label{p22}
A curve $\gamma $ in $\d^2$  given by \eqref{g} is an $\alpha$-catenary if and only if $\operatorname{Re}(\gamma)=\psi$ is an $\alpha$-catenary of $\r^2$ satisfying \eqref{y11}, and  
\begin{equation}\label{eq1}
 z'' + \alpha \frac{y'}{y}(z'+v) + \alpha \frac{z+vx}{y^2} = 0.
\end{equation}
\end{proposition}

\begin{proof} 
 By Proposition \ref{p21}, the real part of $\gamma$ is an $\alpha$-catenary in $\r^2$.  We now compute the Euler-Lagrange equation for  $\mathcal{E}_1$.    The  Lagrangian   is given by 
\begin{equation*}
\begin{split}
L(x,y,w,z,y',w',z')&=\alpha (z+vx)y^{\alpha-1}\sqrt{1+y'^2}+\lambda(w'+y'z'),
\end{split}
\end{equation*}
where $\lambda=\lambda(x)$ is a Lagrange multiplier. We next calculate the corresponding Euler-Lagrange equations using standard techniques. This yields the following system:
\begin{equation} \label{eleqs}
\left\{
\begin{split}
0&=L_y-\frac{d}{dx}(L_{y'}),\\
0&=L_w-\frac{d}{dx}(L_{w'}),\\
0&=L_z-\frac{d}{dx}(L_{z'}) .
\end{split}
\right.
\end{equation}
Each equation in \eqref{eleqs} reduces, respectively, to
\begin{equation} \label{el-re-eqs}
\begin{split} 
 \frac{d}{dx}(\lambda z'+\alpha(z+vx)y^{\alpha-1}\frac{y'}{\sqrt{1+y'^2}})&=  \alpha (\alpha-1)(z+vx)y^{\alpha-2}\sqrt{1+y'^2} \\
 &=c\alpha(\alpha-1)(z+vx)y^{2\alpha-2},\\
 \frac{d}{dx}\lambda&=0,\\
 \frac{d}{dx}(\lambda y') &= \alpha y^{\alpha-1}\sqrt{1+y'^2}=\alpha c y^{2\alpha-1}.
 \end{split} \end{equation}
 
 The second equation in \eqref{el-re-eqs} implies that $\lambda$ is a constant function. The third equation is $\lambda y''=\alpha c y^{2\alpha-1}$. Comparing this  with \eqref{y13}, we deduce that $\lambda=\frac{1}{c}$. Next, using \eqref{y11}, the first equation of \eqref{el-re-eqs} becomes
 $$ \lambda z'' +\alpha(z+vx)y^{\alpha-2}\frac{1}{\sqrt{1+y'^2}} + \alpha(z'+v)y^{\alpha-1}\frac{y'}{\sqrt{1+y'^2}}= 0.$$
Finally, applying \eqref{y12} together with $\lambda=\frac{1}{c}$, we obtain \eqref{eq1}.

\end{proof}
 
We describe the $\alpha$-catenaries for the   particular cases $\alpha\in \{-1,0,1\} $, for which the ODEs \eqref{orto}, \eqref{y11} and \eqref{eq1} admit explicit solutions. In each case, our argument is as follows: we solve \eqref{y11} to obtain $y(x)$, substitute this expression and its derivative into \eqref{eq1} to determine $z(x)$, and then find $w(x)$ from \eqref{orto}. 

The case $\alpha=0$ is immediate.

\begin{proposition} \label{p23}
A curve $\gamma $ in $\d^2$     given by \eqref{g} is a $0$-catenary if and only if 
\begin{equation*}
\begin{split}
y(x)&=\pm\sqrt{c^2-1}x+m,\\
w(x)&=\mp\sqrt{c^2-1}d_1 x+d_3,\\
z(x)&= d_1x+d_2 ,
\end{split}
\end{equation*}
where $ c,m,d_i\in\r$, $c>0$.
 
\end{proposition}

 \begin{proposition} \label{p24}
A curve $\gamma $ in $\d^2$  given by \eqref{g} is a catenary if and only if 
\begin{equation}
\begin{split}\label{s1}
y(x)&=\frac{1}{c}\cosh(cx+m),\\
w(x)&=\frac{v}{c}\cosh(cx+m) + c d_1 x - d_1 \tanh(cx+m) + d_2 \sech(cx+m) + d_3,\\
z(x)&=-vx + d_1 \sech(cx+m) + d_2 \tanh(cx+m),
\end{split}
\end{equation}
where $c, m, v, d_i \in \r$ and $c>0$.
\end{proposition}

\begin{proof}
From Proposition \ref{p21}, it follows that the real part  $\psi(x)=(x,y(x))$ is the standard catenary in $\r^2$. Thus, $y(x) = \frac{1}{c}\cosh(cx+m)$, and consequently $y'(x) = \sinh(cx+m)$, where $c,m\in \r$, $c>0$. Substituting $y$ and $y'$ into   \eqref{eq1}, we obtain 
$$
z'' +  c \tanh(cx+m) (z'+v) + \frac{c^2(z+vx)}{\cosh^2(cx+m)} = 0.
$$
To solve this equation, we use the change of variable $u(x) = z(x) + vx$. It directly follows that $u' = z' + v$ and $u'' = z''$. Hence, the equation transforms into  
$$
u'' + c \tanh(cx+m) u' + c^2 \text{sech}^2(cx+m) u = 0.
$$
This linear equation has general solution
$$
u(x) = d_1 \text{sech}(cx+m) + d_2 \tanh(cx+m),
$$
where $d_1, d_2 \in \r$. Then,  the   expression for $z(x)$ is given by  \eqref{s1}. It remains to find $w(x)$. From the orthogonality condition \eqref{orto}, we have $w' = -y'z'$. Differentiating $z(x)$ gives
$$
z'(x) = -v - c d_1 \text{sech}(cx+m)\tanh(cx+m) + c d_2 \text{sech}^2(cx+m).
$$
Multiplying by $-y'(x) = -\sinh(cx+m)$, we write
$$
w'(x) = v \sinh(cx+m) + c d_1 \tanh^2(cx+m) - c d_2 \tanh(cx+m)\text{sech}(cx+m).
$$
By integrating, we find the expression for $w$ in \eqref{s1}. 
\end{proof}

\begin{proposition}\label{p25}
A curve $\gamma $ in $\d^2$     given by \eqref{g} is a $(-1)$-catenary if and only if 
\begin{equation}
\begin{split}\label{s2}
y(x)&=\sqrt{R^2-(x-m)^2},\\
w(x)&=(v-d_1) y(x) + d_2 (x-m) - d_2 y(x) \arcsin\left(\frac{x-m}{R}\right) + d_3,\\
z(x)&=-vx + d_1(x-m) + d_2 \left( y(x) + (x-m)\arcsin\left(\frac{x-m}{R}\right) \right),
\end{split}
\end{equation}
where $R, m, v, d_i \in \mathbb{R}$ with $R>0$. Furthermore, the maximal domain is given by $|x-m| < R$.
\end{proposition}

\begin{proof}
For $\alpha=-1$, the   curve $\psi$ is a $(-1)$-catenary in $\r^2$, which corresponds to a circular arc. Thus, we have $y(x) = \sqrt{R^2-(x-m)^2}$, and its derivative is $y'(x) = \frac{-(x-m)}{y(x)}$, where $m,R\in \r$, $R>0$. Substituting $\alpha=-1$ into \eqref{eq1}, we have
$$
z'' - \frac{y'}{y}(z'+v) - \frac{z+vx}{y^2} = 0.
$$
We introduce the change of variable $u(x)=z(x)+vx$, obtaining
$$
u'' - \frac{y'}{y} u' - \frac{1}{y^2} u = 0.
$$
Substituting $y$ and $y'$ into this equation yields
$$
u'' + \frac{x-m}{R^2-(x-m)^2} u' - \frac{1}{R^2-(x-m)^2} u = 0,
$$
or equivalently,
$$
(R^2-(x-m)^2) u'' + (x-m) u' - u = 0.
$$
It is straightforward to verify that $u_1(x) = x-m$ is a particular solution to this equation. A second, linearly independent solution can be found by using reduction of order, yielding 
$$u_2(x) = \sqrt{R^2-(x-m)^2} + (x-m)\arcsin\left(\frac{x-m}{R}\right).$$
 Noting that the first term of $u_2$ is exactly $y(x)$, the general solution for $u(x)$ is
$$
u(x) = d_1 (x-m) + d_2 \left( y(x) + (x-m)\arcsin\left(\frac{x-m}{R}\right) \right),
$$
where $d_1, d_2 \in \mathbb{R}$. Reverting the change of variable, we obtain the expression for $z(x)$ in  \eqref{s2}.

Finally, we need to find $w$ from \eqref{orto}. Differentiating $z(x)$ gives
$$
z'(x) = -v + d_1 + d_2 \arcsin\left(\frac{x-m}{R}\right).
$$
Hence
$$w'(x) = \frac{x-m}{\sqrt{R^2-(x-m)^2}} \left( -v + d_1 + d_2 \arcsin\left(\frac{x-m}{R}\right) \right).$$
An integration yields the expression for $w(x)$ in \eqref{s2}.   This completes the proof.
\end{proof}

Observe that for these three values of $\alpha$, the coordinate function $y(x)$ of $\mbox{Re}(\gamma)$ also appears in the coordinates of dual part of $\gamma$. In the following example, we illustrate this particularity clearly.

\begin{example} For $\alpha=1$ and $\alpha=-1$,  choose   the integration constants in \eqref{s1} and \eqref{s2} as $d_1=d_2=d_3=0$. Then,  the dual part of the curve $\gamma$ simplifies   to $\phi(x) =  v(y(x), -x)$.  Thus, the $\alpha$-catenary $\gamma$ takes the   form
$$\gamma(x) = (x, y(x)) + \varepsilon v (y(x), -x),$$
where $y(x) = \frac{1}{c}\cosh(cx+m)$ ($\alpha=1$) or $y(x) = \sqrt{R^2-(x-m)^2}$ ($\alpha=-1$). 
\end{example}

Thanks to these examples, we can extend this result for any $\alpha$ as follows.

\begin{theorem} \label{t27}
Let $\alpha\in\r$. If $\psi(x) = (x, y(x))$ is an $\alpha$-catenary in $\r^2$, then 
\begin{equation}\label{reverse}
\gamma(x) = (x, y(x)) + \varepsilon v (y(x), -x), \quad v\in \r,
\end{equation}
is an $\alpha$-catenary in $\d^2$. 
\end{theorem}

 \begin{proof}
 The result is straightforward. Indeed, we have $w(x)=v y(x)$ and $z(x)=-vx$, $v\in \r$. The orthogonality condition \eqref{orto} is immediately satisfied. Equation \eqref{eq1} for $z$ is also easily verified.

 \end{proof}

This theorem shows a particular way to construct $\alpha$-catenaries in $\d^2$ by using only $\alpha$-catenaries of $\r^2$. 
We can interpret this behavior by saying that the dual part is generated by swapping and scaling the real part. 

\begin{definition} If $\psi(x)$ is an $\alpha$-catenary in $\r^2$,   the $\alpha$-catenary in $\d^2$ defined by \eqref{reverse} is called the reversed catenary.
\end{definition}
The existence of the reversed catenary reflects that there is a certain of symmetry in  the structure of $\alpha$-catenaries in $\d^2$.

\section{A characterization of catenaries in $\d^2$} \label{sec3}

In this section, we give a characterization of catenaries in the dual plane in terms of the curvature and the normal vector of the curve. Our approach is motivated by the corresponding characterization in the Euclidean plane. Let $\psi$ be a curve in $\r^2$. Then $\psi$ is an $\alpha$-catenary in $\r^2$ if and only if its curvature $\kappa_\psi$ satisfies
\begin{equation}
\kappa_\psi=\alpha \frac{\langle N_\psi,\vec{u} \rangle}{\langle \psi,\vec{u}\rangle}, \quad \vec{u}=(0,1),\label{cpsi}
\end{equation}
where $N_\psi$ is the unit normal vector   of   $\psi$. 

We now obtain an analogous characterization for  $\alpha$-catenaries in the dual plane $\d^2$. First, we compute the curvature $\kappa_\gamma$ of a curve $\gamma$ in $\d^2$ given by \eqref{g}. The curvature of admissible curves parametrized by arc-length was introduced in  \cite{lo}.
 
Let $\gamma(s)=\psi(s)+\varepsilon\phi(s)$ be a curve in $\d^2$ parametrized by arc-length. Denote by $\kappa_\psi$ and $N_\psi$ the curvature and the unit normal vector of $\psi$, respectively. Then
\begin{equation}\label{cu0}
\kappa_\gamma=\kappa_\psi+\varepsilon\langle\phi''_s,N_\psi\rangle, \quad \phi''_s:=\frac{d^2\phi}{ds^2},
\end{equation}
where $\kappa_\gamma$ is the curvature of $\gamma$. Using this identity, we derive a corresponding expression for $\kappa_\gamma$ when $\gamma$ is parametrized as in \eqref{g}.

\begin{lemma} Let $\gamma=\gamma(x)$ be an admissible curve in $\d^2$ parametrized by \eqref{g}. Then its curvature $\kappa_\gamma$ is 
\begin{equation}\label{cu}
\begin{split}
\kappa_\gamma&=\kappa_\psi + \varepsilon \frac{z''}{\sqrt{1+y'^2}}\\
&=\frac{y''}{(1+y'^2)^{3/2}}+\varepsilon \frac{z''}{\sqrt{1+y'^2}}.
\end{split}
\end{equation}
\end{lemma}

\begin{proof}
We reparametrize $\gamma$ by arc-length (see \cite[Proposition 2.1]{lo}) and write
$$\gamma(s)=\gamma(f(s))=\psi(f(s))+\varepsilon \phi(f(s)).$$
Notice that the real part of $\kappa_\gamma$ in \eqref{cu0}, i.e. $\mbox{Re}(\kappa_\gamma)=\kappa_\psi $, does not depend on the choice of parametrization. Hence, it remains to determine the dual part of $\kappa_\gamma$. 

As two parameters, $s$ and $x$, are now involved,  we specify the notation for differentiation as follows:
$$
\gamma'=\frac{d\gamma}{dx}, \quad \gamma'_s=\frac{d\gamma}{ds}.
$$ 
Since $\gamma$ is parametrized by arc-length and 
$$\gamma'_s=f'_s(\psi'+\varepsilon \phi'),$$
it follows that
$$f'_s=\frac{1}{\nu},\quad \nu:=\sqrt{1+y'^2}.$$
In addition, the unit normal vector field of $\psi$ is 
$$N_\psi=\frac{(-y',1)}{\nu}.$$
By using \eqref{orto} and applying the chain rule, we compute
\begin{equation*}
\begin{split}
\phi''_s&=f''_s\phi'+f'^2_s\phi''=f''_s(w',z')+f'^2_s(w'',z'')\\
&=z'f''_s(-y',1)+f'^2_s(w'',z'')\\
&=z'f''_s\nu N_\psi+\frac{1}{\nu^2}(w'',z'').
\end{split}
\end{equation*}
Taking the inner product with $N_\psi$, we obtain
\begin{equation}\label{inner}
\langle\phi''_s,N_\psi\rangle=z'f''_s\nu+\frac{1}{\nu^3}(z''-y'w'').
\end{equation}
Differentiating the orthogonality condition \eqref{orto}, we have $w''=-y''z'-y'z''$. Moreover, 
$$f''_s=\left(\frac{1}{\nu}\right)'f'_s=-\frac{y'y''}{\nu^4}.$$
Substituting these expressions into \eqref{inner} yields
$$\langle\phi''_s,N_\psi\rangle=-\frac{z'y'y''}{\nu^3}+\frac{1}{\nu^3}(z''\nu^2 +z'y'y'')=\frac{z''}{\nu}.$$
This completes the proof.
\end{proof}

We now establish the dual version of \eqref{cpsi}. 

\begin{theorem}\label{t32}
Let $\gamma(x)=\psi(x)+\varepsilon \phi(x)$ be an admissible curve in $\d^2$. Then $\gamma$ is an $\alpha$-catenary if and only if
its curvature $\kappa_\gamma$ satisfies
\begin{equation}\label{chara}
  \kappa_\gamma=\alpha \frac{\langle N_\gamma,\mathbf{u}\rangle}{\langle \gamma,\mathbf{u}\rangle},
\end{equation}
where $N_\gamma$ is the unit normal vector of $\gamma$ and $\mathbf{u}\in \d^2$ is a fixed unit vector.
\end{theorem}

\begin{proof}
After a change of coordinates, we may assume that $\mathbf{u}=(0,1)+\varepsilon(v,0)$, where $v\in \r$. With this choice of coordinates, we consider the parametrization  \eqref{g} of $\gamma$. A direct computation gives
$$\langle \gamma,\mathbf u\rangle = y +\varepsilon(vx+z).$$
For the computation of $N_\gamma$, we can use \cite[Proposition 2.8]{lo}, or proceed directly as follows. Let $T_\gamma$ and $T_\psi$ be the unit tangent vectors of $\gamma$ and $\psi$, respectively. By the previous lemma, 
$$T_\gamma=\gamma'_s=f'_s(\psi'+\varepsilon\phi')=T_\psi+\varepsilon z' N_\psi.$$
Therefore, the unit normal vector field of $\gamma$ is
$$N_\gamma=  N_\psi - \varepsilon z' T_\psi,$$
where
$$N_\psi= \frac{1}{\nu}(-y', 1),\quad T_\psi=\frac{1}{\nu}(1,y').  $$
The numerator on the right hand-side of \eqref{chara} is 
\begin{equation*}
\begin{split}
\langle N_\gamma,\mathbf{u}\rangle&=\langle N_\psi,(0,1)\rangle+\varepsilon\left(\langle N_\psi,(v,0)-z'\langle T_\psi,(0,1)\rangle\right)\\
&= \frac{1}{\nu}\left(1 - \varepsilon y'(v+z')\right).
\end{split}
\end{equation*}
 We now use the division operation for dual numbers,
 $$\frac{a+b\varepsilon}{c+d\varepsilon}=\frac{a}{c}+\frac{bc-ad}{c^2}\varepsilon.$$
Moreover, identity \eqref{cpsi} takes the form
 $$\kappa_\psi=\frac{\alpha}{y\nu}.$$
Using \eqref{eq1}, we rewrite
$$z''=-\frac{\alpha}{y^2}(yy'(v+z')+ vx+z).$$
Therefore,
\begin{equation*}
\begin{split}
 \alpha\frac{\langle N_\gamma,\mathbf u\rangle}{\langle \gamma,\mathbf u\rangle} &=\frac{\alpha}{y\nu}\left(
 1-\frac{\varepsilon}{y}(yy'(v+z')+vx+z)\right)\\
 &=\frac{\alpha}{y\nu}-\varepsilon\frac{\alpha}{y^2\nu}(yy'(v+z')+vx+z)\\
 &=\kappa_\psi+\varepsilon\frac{z''}{\nu}=\kappa_\gamma,
 \end{split}
 \end{equation*}
 where the last identity follows exactly from \eqref{cu}. 
\end{proof}
 
As a conseqeunce, we have the following.

\begin{corollary}
The curvature of an $\alpha$-catenary in $\r^2$ coincides with the curvature of its reversed catenary in $\d^2$. In particular, the curvature of a reversed catenary is pure real-valued. 
\end{corollary}

\begin{proof}
By the definition of the reversed catenary $\gamma$ given in \eqref{reverse}, we have $z(x) = -vx$. Thus, $z''=0$. Substituting this into \eqref{cu}, we obtain $\kappa_\gamma=\kappa_\psi$, completing the proof.
\end{proof}


\subsection*{Data Availability Statement}
No data are available for this study.

  \section*{Acknowledgements}
Rafael L\'opez  has been partially supported by MINECO/MICINN/FEDER grant no. PID2023-150727NB-I00,  and by the ``Mar\'{\i}a de Maeztu'' Excellence Unit IMAG, reference CEX2020-001105- M, funded by MCINN/AEI/10.13039/ 501100011033/ CEX2020-001105-M.


\end{document}